\documentclass[leqno,11pt]{amsart}
\usepackage{amssymb, amsmath,amsmath,latexsym,amssymb,amsfonts,amsbsy, amsthm,esint}
\usepackage{color}
\usepackage{graphicx}
\usepackage{tikz}
\usetikzlibrary{arrows, calc, decorations.markings,intersections, patterns,patterns.meta}
\usepackage{caption}
\usepackage{subcaption}
\usepackage{hyperref}

\usepackage{pgfplots}
\pgfplotsset{compat=1.15}
\usepackage{mathrsfs}
\usepackage[foot]{amsaddr}

\let\pa=\partial
\let\al=\alpha

\let\g=\gamma
\let\d=\delta

\let\lam=\lambda

\let\s=\sigma
\let\t=\theta
\let\f=\frac

\let\D=\Delta

\let\Om=\Omega

\let\e=\varepsilon
\let\pa=\partial

\let\ri=\rightarrow
\let\na=\nabla

\newcommand{\beq}{\begin{equation}}
\newcommand{\eeq}{\end{equation}}
\newcommand{\beqo}{\begin{equation*}}
\newcommand{\eeqo}{\end{equation*}}
\newcommand{\ben}{\begin{eqnarray}}
\newcommand{\een}{\end{eqnarray}}
\newcommand{\beno}{\begin{eqnarray*}}
\newcommand{\eeno}{\end{eqnarray*}}

\newcommand{\dist}{\mathrm{dist}}

\newcommand{\BR}{\mathbb{R}}

\newcommand{\cl}{\mathcal{L}^1}

\newtheorem{theorem}{Theorem}[section]
\newtheorem{definition}[theorem]{Definition}
\newtheorem{lemma}[theorem]{Lemma}
\newtheorem{proposition}[theorem]{Proposition}

\theoremstyle{remark}

\newtheorem{case}{Case}
\newtheorem{rmk}{Remark}[section]

\newcommand{\mres}{\mathbin{\vrule height 1.6ex depth 0pt width
0.13ex\vrule height 0.13ex depth 0pt width 1.2ex}}

\begin{document}
\title[Triple junction problem]{On the triple junction problem with general surface tension coefficients}
\author{Nicholas D. Alikakos$^1$}
\address{$^1$Department of Mathematics, University of Athens (EKPA), Panepistemiopolis, 15784 Athens, Greece}
\email{nalikako@math.uoa.gr}

\author{Zhiyuan Geng$^2$}
\address{$^2$Basque Center for Applied Mathematics, Alameda de Mazarredo 14
48009 Bilbao, Bizkaia, Spain}
\email{zgeng@bcamath.org}

\date{\today}

\begin{abstract}
    We investigate the Allen-Cahn system
    \begin{equation*}
    \Delta u-W_u(u)=0,\quad u:\mathbb{R}^2\rightarrow\mathbb{R}^2,
    \end{equation*}
    where $W\in C^2(\mathbb{R}^2,[0,+\infty))$ is a potential with three global minima. We establish the existence of an entire solution $u$ which possesses a triple junction structure. The main strategy is to study the global minimizer $u_\varepsilon$ of the variational problem 
    \begin{equation*}
    \min\int_{B_1} \left( \f{\varepsilon}{2}|\nabla u|^2+\f{1}{\varepsilon}W(u) \right)\,dz,\ \ u=g_\varepsilon \text{ on }\partial B_1.
    \end{equation*}
    The point of departure is an energy lower bound that plays a crucial role in estimating the location and size of the diffuse interface. We do not impose any symmetry hypotheses on the solution or on the potential. 
\end{abstract}


\maketitle

\section{Introduction}

We study the existence of an entire, bounded, minimizing solution to the system 
\beq\label{main equation}
\D u-W_u(u)=0, \quad u:\BR^2\ri\BR^2,
\eeq
where $W$ is a triple-well potential with three global minima. For the potential $W$ we assume
\vspace{3mm}
\begin{enumerate}
    \item[(H1).]  $W\in C^2(\BR^2;[0,+\infty))$, $\{z: \,W(z)=0\}=\{a_1,a_2,a_3\}$, $W_u(u)\cdot u >0$ if $|u|>M$ and 
    \beqo
     c_2|\xi^2| \geq \xi^TW_{uu}(a_i)\xi\geq  c_1|\xi|^2,\; i=1,2,3.
    \eeqo 
    for some positive constants $c_1<c_2$ depending on $W$. 
    \item[(H2).] For $i\neq j$, $i,j\in \{1,2,3\}$, let $U_{ij}\in W^{1,2}(\BR,\BR^2)$ be the 1D minimizer with action
\beqo
\sigma_{ij}:=\min \int_{\BR}\left(\f12|U_{ij}'|^2+W(U_{ij})\right)\,d\eta, \quad \lim\limits_{\eta\ri-\infty}U_{ij}(\eta)=a_i,\ \lim\limits_{\eta\ri+\infty}U_{ij}(\eta)=a_j.
\eeqo
We assume that the \emph{surface tension coefficients} $\sigma_{ij}$ satisfy
\beq\label{cond on sigma}
\sigma_{ij}<\sigma_{ik}+\sigma_{jk}\quad \text{ for }i\neq j,k\text{ with }i,j,k\in\{1,2,3\}.
\eeq
\end{enumerate}
We refer the reader to \cite[Proposition 2.6 \& Proposition 2.12]{afs-book} for the sufficiency and necessity of \eqref{cond on sigma} for the existence of all three $U_{ij}$ above.

Note that \eqref{main equation} is the Euler-Lagrange equation corresponding  to 
\begin{equation}\label{energy functional}
J(u,\Omega):=\int_\Om \left( \f12|\na u|^2+W(u) \right)\,dz,\ \ \forall \text{ bounded open set } \Om\subset \BR^2. 
\end{equation}

 Our definition of the \emph{minimizing solution} is as follows.
\begin{definition}
$u$ is a \emph{minimizing solution} of \eqref{main equation} in the sense of De Giorgi if 
\begin{equation}
    J(u,\Om)\leq J(u+v,\Om), \quad \forall \text{ bounded open set } \Om\subset \BR^2, \ \forall  v\in C_0^1(\Om).
\end{equation}
\end{definition}

The solution we study here is closely related to the minimal partition problem on the plane. More precisely, we define
\beqo
\mathcal{P}=\{D_1,D_2,D_3\}, 
\eeqo
which is a partition of the plane into three sectors with angles $\al_1,\,\al_2,\,\al_3$ respectively. We call $\pa \mathcal{P}=\pa D_1\cup\pa D_2\cup\pa D_3$ a \emph{triod} and its center a \emph{triple junction}. Moreover, these angles satisfy Young's law, that is
\begin{align}
\label{sum of angles} &\sum_{i=1}^3\al_i=2\pi,\quad \theta_i\in (0,\pi) \text{ for every }i\in\{1,2,3\},\\
\label{relation of angles} & \qquad \f{\sin{\al_{1}}}{\s_{23}}=\f{\sin{\al_2}}{\s_{13}}=\f{\sin{\al_3}}{\s_{12}}.
\end{align}
It is well known that $\mathcal{P}$ is a (complete) minimizing partition of $\BR^2$ into three phases and $\pa\mathcal{P}$ is a minimal cone. $\mathcal{P}$ is minimizing in the sense that for any $\Omega\subset \BR^2$, $\mathcal{P}\mres \Omega=\{D_i\cap\Om\}_{i=1}^3 $ is a solution of the following variational problem 
\beqo
\begin{split}
&\qquad\qquad\quad \min \sum\limits_{i<j} \s_{ij} \mathcal{H}^1(\pa (A_i\cap \Om) \cap \pa(A_j\cap\Om)),\\
&\mathcal{A}=\{A_i\}_{i=1}^3 \text{ is a 3--partition of }\mathbb{R}^2 \text{ and } \mathcal{P}\mres(\mathbb{R}^2\setminus \Omega)= \mathcal{A}\mres(\mathbb{R}^2\setminus \Omega). 
\end{split}
\eeqo
For the significance of the triangle inequality \eqref{cond on sigma} for the minimal partitioning problem, we refer to \cite{white1996existence}.

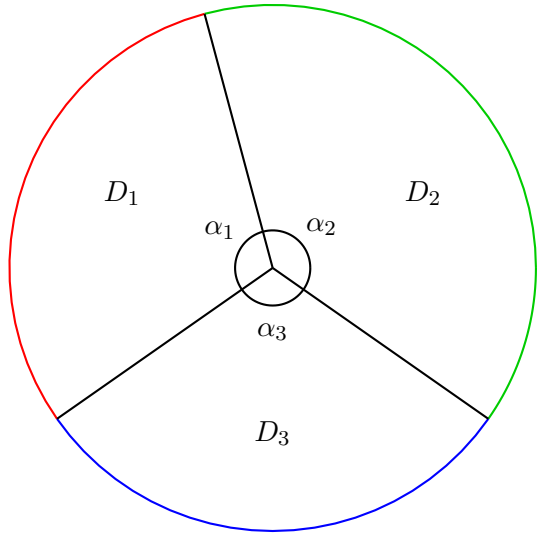
\begin{figure}[ht]
\begin{tikzpicture}[thick]
\draw [black!20!green,domain=0:105] plot ({3.5*cos(\x)}, {3.5*sin(\x)}) node at (2,1)[black]{$D_2$};
\draw [domain=0:105] plot ({0.5*cos(\x)}, {0.5*sin(\x)}) node at (0.65,0.55) {$\al_2$};
\draw [domain=105:215] plot ({0.5*cos(\x)}, {0.5*sin(\x)}) node at (-0.7,0.5) {$\al_1$};
\draw [domain=215:325] plot ({0.5*cos(\x)}, {0.5*sin(\x)}) node at (0,-0.85) {$\al_3$};
\draw [domain=325:360] plot ({0.5*cos(\x)}, {0.5*sin(\x)});
\draw [black!20!green,domain=325:360] plot ({3.5*cos(\x)}, {3.5*sin(\x)});
\draw [red,domain=105:215] plot ({3.5*cos(\x)}, {3.5*sin(\x)}) node at (-2,1)[black]{$D_1$};
\draw [blue,domain=215:325] plot ({3.5*cos(\x)}, {3.5*sin(\x)}) node at (0,-2.2)[black]{$D_3$};
\draw  ({3.5*cos(105)},{3.5*sin(105)})--(0,0);
\draw ({3.5*cos(215)},{3.5*sin(215)})--(0,0);
\draw ({3.5*cos(325)},{3.5*sin(325)})--(0,0);
\end{tikzpicture}
\caption{The partition $\mathcal{P}=\{D_1,D_2,D_3\}$ with opening angles $\al_1,\al_2,\al_3$ satisfying \eqref{sum of angles} \& \eqref{relation of angles}, and $\Omega$ a ball centered at the junction}
\label{pic2}
\end{figure}

Now we state our main theorem. 
\begin{theorem}\label{main theorem}
Fix $\g< \min\{\g_0,\min\limits_{i,j\in\{1,2,3\}} \f12 \dist(a_i,a_j), \sqrt{\f{\sigma}{20C_W}}\}$. There is a constant $C_0$ which depends on $\g,W$ such that under the hypotheses (H1), (H2),  there exists an entire, bounded minimizing solution of \eqref{main equation} with the following triple junction structure. 
\begin{enumerate}
    \item[a.] For every $r>0$, there exists a point $P(r)$ such that $|u(P(r))-a_1|\leq \g$.
    \item[b.] There exists $\{Q_j\}_{j=1}^\infty\cup \{R_j\}_{j=1}^\infty$ such that 
    \beqo
    \begin{split}
    &\vert u(Q_j)-a_2\vert\leq \g,\quad \vert u(R_j)-a_3\vert\leq \g,\\
    &\dist(Q_j, 0)\leq 32jC_0,\quad \dist(R_j, 0)\leq 32jC_0.
    \end{split}
    \eeqo
    The pairwise distance of any two points from $\{Q_j\}_{j=1}^\infty\cup \{R_j\}_{j=1}^\infty$ is larger than or equal to $6C_0$.
    \item[c.] For each $Q_j$, there exists $P_j$ such that 
    \beqo
    \dist(Q_j,P_j)\leq C_0,\quad |u(P_j)-a_1|\leq \g;
    \eeqo
    This property also holds for $\{R_j\}$.
    \item[d.]  For any sequence $r_k\ri+\infty$ one can extract a subsequence, still denoted by $\{r_k\}$,  such that 
\beq\label{convergence in thm 1.3}
u(r_k x)\rightarrow u_0(x)\ \text{ in }L^1_{loc}(\BR^2),
\eeq
where $u_0(x)=\sum\limits_{i=1}^3a_i \mathcal{\chi}_{D_i}$. Here $\chi$ is the characteristic function. $\mathcal{P}=\{D_1,D_2,D_3\}$ provides a minimal partition of $\BR^2$ into three sectors with angles $\al_i\;(i=1,2,3)$ and $\pa \mathcal{P}$ is a triod centered at $0$. The partition $\mathcal{P}$ may depend on $\{r_k\}$. Also there exists a sequence $\{r'_k\}_{k=1}^\infty$, such that for each $\xi\in D_j\, (j=1,2,3)$, $|\xi|=1$, 
\beq
\lim\limits_{k\ri\infty}\f{1}{r_k'}\int_0^{r_k'} u(s\xi)\,ds=a_j,   \tag{\ref*{convergence in thm 1.3}'} \label{convergence in thm 1.3 along a rays}
\eeq
and the convergence is uniform for $\xi$ in compact sets of $\mathbb{S}^1\setminus \pa\mathcal{P}$.
\end{enumerate}
\end{theorem}

In our  paper \cite{alikakos2022triple}, Theorem \ref{main theorem} was proved for the specific choice of $\sigma_{12}=\sigma_{13}=\sigma_{23}$. The objective of the current note is to extend this result to the general case where condition \eqref{cond on sigma} is satisfied. We will follow the proof in \cite{alikakos2022triple} step by step, with the main difference being the more complex calculation required to estimate the energy lower bound in Proposition \ref{prop: lower bound epsilon 1/3}. This calculation is presented in detail in Section \ref{sec:lower bound} and in the Appendices. For the rest of the proof, we will refer directly to \cite{alikakos2022triple} and only mention the necessary modifications.

\section{Preliminaries} \label{sec:preliminary}

Throughout the paper we denote by $z=(x,y)$ a 2D point and by $B_1=B_1(0)$ the unit 2D ball centered at the origin. We recall the following basic results.
\begin{lemma}[Lemma 2.1 in \cite{AF}]\label{lemma: potential energy estimate}
The hypotheses on $W$ imply the existence of $\delta_W>0$, and constants $c_W,C_W>0$ such that
\beqo
\begin{split}
&|u-a_i|=\delta\\
\Rightarrow & \ \f12 c_W\delta^2\leq W(u)\leq \f12 C_W \delta^2,\quad \forall \delta<\delta_W,\ i=1,2,3.
\end{split}
\eeqo
Moreover if $\min\limits_{i=1,2,3} |u-a_i|\geq \delta$ for some $\delta<\delta_W$, then $W(u)\geq \f12 c_W\delta^2.$
\end{lemma}

\begin{lemma}[Lemma 2.3 in \cite{AF}]\label{lemma: 1D energy estimate}
Take $i\neq j \in \{1,2,3\}$, $\d <\d_W$ and $s_+>s_+$ be two real numbers. Let $v:(s_-,s_+)\ri \BR^2$ be a smooth map that minimizes the energy functional 
\beqo
J_{(s_-,s_+)}(v):=\int_{s_-}^{s_+} \left(\f12|\na v|^2+W(v)\right)\,dx 
\eeqo
subject to the boundary condition 
\beqo
|v(s_-)-a_i|=|v(s_+)-a_j|=\delta.
\eeqo
Then
\beqo
J_{(s_-,s_+)}(v)\geq \sigma_{ij}-C_W\delta^2,
\eeqo
for some constant $C_W$ depending only on the potential $W$.
\end{lemma}

For most of the paper, we consider the variational problem 
\beq\label{epsilon energy}
\min\int_{B_1}\left( \f{\e}{2}|\na u|^2+\f1\e W(u)\right)\,dz,
\eeq
where $W(u)$ satisfies Hypotheses (H1) and (H2). We denote by $u_\e \in W^{1,2}(B_1,\BR^2)$ a global minimizer of the functional \eqref{epsilon energy} with respect to the boundary condition in polar coordinates 
$$
u_\e(1,\theta)=g_\e(\t)\text{ on }\pa B_1,
$$
where $g_\e:[0,2\pi)$ is given by
\begin{equation}\label{def of g_eps}
g_\e(\theta):=\begin{cases}
a_2+g_0(\f{\t-(\f{\pi}{2}+\f{\al_2-\al_1}{2}-c_0\e)}{2c_0\e})(a_1-a_2), &  \t\in [\f{\pi}{2}+\f{\al_2-\al_1}{2}-c_0\e, \f{\pi}{2}+\f{\al_2-\al_1}{2}+c_0\e)\\
a_1,& \t\in[\f{\pi}{2}+\f{\al_2-\al_1}{2}+c_0\e, \f32\pi-\f{\al_3}{2}-c_0\e),\\
a_1+g_0(\f{\t-(\f32\pi-\f{\al_3}{2}-c_0\e)}{2c_0\e})(a_3-a_1), &  \t\in [\f32\pi-\f{\al_3}{2}-c_0\e, \f32\pi-\f{\al_3}{2}+c_0\e)\\
a_3,& \t\in[\f32\pi-\f{\al_3}{2}+c_0\e, \f{3}{2}\pi+\f{\al_3}{2}-c_0\e),\\
a_3+g_0(\f{\t-(\f{3}{2}\pi+\f{\al_3}{2}-c_0\e)}{2c_0\e})(a_2-a_3), &  \t\in [\f{3}{2}\pi+\f{\al_3}{2}-c_0\e, \f{3}{2}\pi+\f{\al_3}{2}+c_0\e),\\
a_2,& \t\in[\f{3}{2}\pi+\f{\al_3}{2}+c_0\e, 2\pi)\cap[0,\f{\pi}{2}+\f{\al_2-\al_1}{2}-c_0\e),
\end{cases}
\end{equation}
where $g_0: [0,1]\ri [0,1]$ is a strictly increasing smooth function that satisfies $g_0(0)=0$, $g_0(1)=1$ and $|g_0'(x)|\leq 2$. Here $\al_1,\al_2,\al_3$ satisfy \eqref{sum of angles} and \eqref{relation of angles}. Without loss of generality we assume $\al_2\geq\al_1$ and consequently by \eqref{relation of angles} and $\al_1+\al_2>\pi$ we have
\beqo
\sigma_{23}\geq \sigma_{13}.
\eeqo

We note that our choice of coordinate axes is such that the $y$-axis bisects $\al_3$. The other perhaps more natural option of choosing the $y$-axis at one of the transitions on the boundary leads to complications. Moreover, from the definition there exists a positive constant $M$ such that $|g_\e|\leq M$. 

The energy $J_\e(u_\e)$ satisfies the following upper bound.
\begin{lemma}\label{lemma:upper bound}
There is a constant $C=C(W)$ such that  
\begin{equation}\label{upper bound}
\int_{B_1}\left\{  \f{\e}{2}|\na u_\e|^2+\f{1}{\e}W(u_\e) \right\}\,dz\leq \s_{12}+\s_{23}+\s_{13} +C\e.
\end{equation}
\end{lemma}

The upper bound can be proved by estimating the energy of an explicitly constructed energy competitor. We refer the reader to \cite[Appendix A]{alikakos2022triple} for the detailed proof of the case where $\sigma_{ij}$ are all equal. The arguments can be applied to the general potential case easily.

In addition, we can control $|u_\e|$ and $|\na u_\e|$ thanks to the smoothness assumption of $W$ and standard elliptic regularity theory. The proof is omitted.

\begin{lemma}
Let $u_\e$ minimize the functional \eqref{epsilon energy} with the boundary condition $u_\e=g_\e$ on $\pa B_1$. There is a constant $M$ independent of $\e$, such that  
\begin{equation}\label{gradient bound}
|u_\e(z)|\leq M,\quad |\na u_\e(z)|\leq\f{M}{\e}, \quad \forall z\in B_1. 
\end{equation}
\end{lemma}

\section{Lower bound for $J_\e(u_\e)$} \label{sec:lower bound}

\begin{proposition}\label{prop: lower bound epsilon 1/3}(weak lower bound)
There exist constants $C$ and $\e_0$, such that for any $\e\leq \e_0$, it holds 
\beq\label{lower bound e 1/3}
\int_{B_1} \left(\f{\e}{2}|\na u_\e|^2+\f{1}{\e}W(u_\e)\right)\,dz\geq \sigma_{12}+\sigma_{13}+\sigma_{23}-C_1\e^{\f13}.
\eeq 

\end{proposition}

\begin{proof}
We follow essentially \cite[Proposition 3.1]{alikakos2022triple} with the necessary modifications accounting for the unequal surface tensions. For the sake of convenience we write $u_\e=u$ throughout the proof. Set the family of horizontal line segments $\g_y$ for $y\in [-\cos{\f{\al_3}{2}},1]$ as
\beqo
\g_y:=\{(x,y):\; x\in\mathbb{R}\}\cap B_1.
\eeqo

Then we define functions $\lam_1(y)$, $\lam_2(y)$, $\lam_3(y)$ for $y\in [-\cos{\f{\al_3}{2}},1]$,
\beqo
\lam_i(y):=\mathcal{L}^1(\g_y\cap \{|u(x,y)-a_i|<\e^{\f16}\}),\quad i\in \{1,2,3\}.
\eeqo
Here $\mathcal{L}^1$ denotes the 1-dimensional Lebesgue measure. Then by the boundary condition we know for any $y\in [-\cos{\f{\al_3}{2}}+c_0\e, \cos{\f{\al_2-\al_1}{2}}-c_0\e)$, it holds that $\lam_1(y)>0$ and $\lam_2(y)>0$. 

Define the following quantities.

\begin{align*}
&y^*:=\min\{y\in [-\cos{\f{\al_3}{2}}+c_0\e,1]:\lam_1(y)+\lam_2(y)\geq \cl(\g_y)-\e^{\f13}\},\\
&\zeta(x):=\min\{y^*,\sqrt{1-x^2}\},\\
&K:=\{x\in[-\sin{\f{\al_3}{2}}+c_0\e, \sin{\f{\al_3}{2}}-c_0\e]:\  |u((x,\zeta(x)))-a_i|<\e^{\f16},\ i=1 \text{ or } 2\},\\
&M:=\{y\in [-\cos{\f{\al_3}{2}}+c_0\e,y^*]: \lam_3(y)> 0\},\\
&K_1^+:=\{x\in[-\sin{\f{\al_2-\al_1}{2}}+c_0\e,\sin{\f{\al_3}{2}}-c_0\e]:\ |u((x,\zeta(x)))-a_1|<\e^{\f16}\},\\
&K_2^-:=\{x\in[-\sin{\f{\al_3}{2}}+c_0\e,-\sin{\f{\al_2-\al_1}{2}}-c_0\e]:\ |u((x,\zeta(x)))-a_2|<\e^{\f16}\},\\
&\mu_1:=\cl(K_1^+),\\
&\mu_2:=\cl(K_2^-),\\
&\Om_1:=\{z=(x,y)\in B_1: \ y\geq y^*\},\\
&\Om_2:=\{z=(x,y)\in B_1: \ y< y^*\}.
\end{align*}

We consider two cases according to the value of $y^*$.

\begin{case}
\textbf{$y^*\leq \cos{\f{\al_2-\al_1}{2}}-c_0\e$.}  In $\Om_1$, we split $W(u)$ into two parts
\beqo
W=\sin^2\theta_1 W+\cos^2\theta_1 W,
\eeqo
for some $\theta_1\in[0,\f{\pi}{2}]$ to be determined later. 

Let $x_1=\min\{\sqrt{1-|y^*|^2}, \sin{\f{\al_3}{2}}-c_0\e\}$. For any $x\in[-x_1,x_1]$, we consider the vertical line $\{(x,y):y\in\mathbb{R}\}$. Let $A(x)$ denote the point $(x,\sqrt{1-x^2})$ and $B(x)$ denote the point $(x,y^*)$, then $A(x),\,B(x)$ are the two intersections of $\{(x,y):y\in\BR\}$ with $\pa\Om_1$. By the boundary condition and the definition, if $x\in K_1^+$, then $u(A(x))=a_2$ and $|u(B(x))-a_1|\leq \e^{\f16}$. Similarly if $x\in K_2^-$, then $u(A(x))=a_1$ and $|u(B(x))-a_2|<\e^{\f16}$. Therefore we can utilize Lemma \ref{lemma: 1D energy estimate} to estimate the energy in $\Om_1$ in the vertical direction:
\begin{equation}\label{ene:om1,vertical}
    \begin{split}
        &\int_{\Om_1}\left(\f{\e}2|\pa_y u|^2+\f{\sin^2\theta_1}{\e} W(u)\right)\,dxdy\\
        \geq & \int_{K_1^+\cup K_2^-} \int_{y^*}^{\sqrt{1-x^2}} \left(\f{\e}2|\pa_y u|^2+\f{\sin^2\theta_1}{\e} W(u)\right)\,dydx\\
        \geq & \int \int_{K_1^+\cup K_2^-} \sin\theta_1\left(\sigma_{12}-C\e^{\f13}\right)\,dx\\
        \geq &\sin\theta_1(\mu_1+\mu_2)(\sigma_{12}-C\e^{\f13}).
    \end{split}
\end{equation}

For $y\in(y^*,\cos{\f{\al_2-\al_1}{2}}-c_0\e)$, the horizontal line $\{(x,y):x\in\BR\}$ intersects with $\pa \Om_1$ at two points, where $u$ takes the value $a_1,a_2$ respectively. Therefore we have in $\Om_1$ in the horizontal direction
\begin{equation}\label{ene:om1,horizontal}
    \begin{split}
        &\int_{\Om_1}\left(\f{\e}2|\pa_x u|^2+\f{\cos^2\theta_1}{\e} W(u)\right)\,dxdy\\
        \geq & \int_{y^*}^{\cos{\f{\al_2-\al_1}{2}}-c_0\e}\int_{\g_y} \left(\f{\e}2|\pa_x u|^2+\f{\cos^2\theta_1}{\e} W(u)\right)\,dxdy\\
        \geq & \int_{y^*}^{\cos{\f{\al_2-\al_1}{2}}-c_0\e} \cos\theta_1\sigma_{12}\,dy\\
        =&\cos\theta_1\sigma_{12}(\cos{\f{\al_2-\al_1}{2}}-c_0\e-y^*).
    \end{split}
\end{equation}

Adding \eqref{ene:om1,vertical} and \eqref{ene:om1,horizontal} gives 
\begin{equation}\label{ene:om1}
    \begin{split}
        &\int_{\Om_1}\left(\f{\e}2|\na u|^2+\f{1}{\e}W(u)\right)\,dxdy\\
        \geq & \sin\theta_1(\mu_1+\mu_2)(\sigma_{12}-C\e^{\f13})+\cos\theta_1\sigma_{12}(\cos{\f{\al_2-\al_1}{2}}-c_0\e-y^*).
    \end{split}
\end{equation}
Now maximizing with respect to $\theta_1$ we obtain
\begin{equation}\label{ene:om1 final}
\begin{split}
    &\int_{\Om_1}\left(\f{\e}2|\na u|^2+\f{1}{\e}W(u)\right)\,dxdy\\
        \geq &\sigma_{12}\sqrt{(\mu_1+\mu_2)^2+(\cos{\f{\al_2-\al_1}{2}}-y^*)^2}-C\e^{\f13},
\end{split}
\end{equation}
where we have also utilized Cauthy-Schwarz to get the last inequality.

In $\Om_2$, we split $W(u)$ as 
\beqo
W=\sin^2\theta_2W+\cos^2\theta_2W, \quad \text{for some }\theta_2\in[0,\f{\pi}{2}].
\eeqo

For any $x\in[-\sin{\f{\al_3}{2}}+c_0\e, \sin{\f{\al_3}{2}}-c_0\e]$, $A(x)=(x,\zeta(x))$ and $B(x)=(x,-\sqrt{1-x^2})$ are the two intersections of $\{(x,y):y\in\BR\}$ with $\pa\Om_2$. From \eqref{def of g_eps} we know that $u(B(x))=a_3$. For $A(x)$ we have 
\beqo
\begin{cases}
|u(A(x))-a_1|<\e^{\f16}, &x\in K_1^+\cup\left(K\cap[-\sin{\f{\al_3}{2}}+c_0\e,-\sin{\f{\al_2-\al_1}{2}}-c_0\e)\setminus K_2^-\right),\\
|u(A(x))-a_2|<\e^{\f16}, &x\in  K_2^-\cup\left(K\cap[-\sin{\f{\al_2-\al_1}{2}}+c_0\e,\sin{\f{\al_3}{2}}-c_0\e)\setminus K_1^+\right).
\end{cases}
\eeqo
Note that since $\sin\f{\al_3}{2}>\sin{\f{\al_2-\al_1}{2}}$, the interval $(-\sin{\f{\al_3}{2}}+c_0\e,-\sin{\f{\al_2-\al_1}{2}}-c_0\e)$ is well-defined when $\e$ is sufficiently small. Therefore utilizing Lemma \ref{lemma: 1D energy estimate} we can estimate the energy in $\Omega_2$ in the vertical direction as follows:
\beq\label{ene: om2 vertical}
\begin{split}
    &\int_{\Om_2}\left( \f{\e}2|\pa_y u|^2+\f{\sin^2\theta_2}{\e} W(u) \right)\,dxdy\\
    \geq & \int_{x\in K} \int_{-\sqrt{1-x^2}}^{\zeta(x)} \left( \f{\e}2|\pa_y u|^2+\f{\sin^2\theta_2}{\e} W(u) \right)\,dydx\\
    \geq &(\sin{\f{\al_3}{2}}-\sin{\f{\al_2-\al_1}{2}}-2c_0\e-\e^{\f13}-\mu_2+\mu_1)\sin\theta_2 \sigma_{13}\\
    &\qquad+ (\sin{\f{\al_3}{2}}+\sin{\f{\al_2-\al_1}{2}}-2c_0\e-\e^{\f13}-\mu_1+\mu_2)\sin\theta_2\sigma_{23}\\
    \geq & \sin\theta_2\left[\sin{\f{\al_3}{2}}(\sigma_{13}+\sigma_{23})+(\mu_2-\mu_1+\sin{\f{\al_2-\al_1}{2}})(\sigma_{23}-\sigma_{13})\right]-C\e^{\f13},
\end{split}
\eeq
where in the last inequality the constant $C$ only depends on $W$. 

On the domain $\Om_2$, we claim that there exists a constant $C$ such that 
\begin{equation*}
    \mathcal{L}^1([-\cos{\f{\al_3}{2}}+c_0\e,y^*]\setminus M)<C\e^{\f13}.
\end{equation*}

Indeed, we set 
\beqo
S:=\{y\in[-\cos{\f{\al_3}{2}}+c_0\e,y^*]: \lam_3(y)=0\}=[-\cos{\f{\al_3}{2}}+c_0\e,y^*]\setminus M.
\eeqo
For any $y\in S$, the definitions of $y^*$ and $S$ imply that $\lam_1(y)+\lam_2(y)+\lam_3(y)<\mathcal{L}^1(\g_y)-\e^{\f13}$, i.e.
\begin{equation*}
    \cl(\{x \in [-\sqrt{1-y^2},\sqrt{1-y^2}]: |u(x,y)-a_i|>\e^{\f16},\ \forall i\})>\e^{\f13}.
\end{equation*}
From the energy upper bound \eqref{upper bound} and Lemma \ref{lemma: potential energy estimate} we get
\beqo
\begin{split}
    &C\geq \f{1}{\e}\int_S\int_{\g_y} W(u)\ dxdy\geq  \f{c_W}{2\e}\cl(S)\e^{\f13}\e^{\f13}\\
    &\Rightarrow \cl(S)\leq C\e^{\f13},
\end{split}
\eeqo
for some constant $C$ depending on $W$, and thus the claim is established. We note that the argument above utilized just that the energy is bounded by a constant. The precise upper bound in \eqref{upper bound} however plays a role in several places later.

Now for any $y\in M$, it holds that
 \beqo
 u(-\sqrt{1-y^2},y)=a_1,\ u(\sqrt{1-y^2},y)=a_2,\ \exists (x_0,y)\in \g_y \text{ s.t. }|u(x_0,y)-a_3|<\e^{\f16}.
 \eeqo
We have therefore in $\Om_2$ in the horizontal direction: 
\begin{equation}\label{ene: om2 horizontal}
     \begin{split}
     &\int_{\Om_2} \left( \f\e2|\pa_{x}u|^2+\f{\cos^2\theta_2}{\e}W(u) \right)\,dxdy\\
     \geq &\int_{y\in M} \int_{\g_y}  \left(  \f\e2|\pa_{x}u|^2+\f{\cos^2\theta_2}{\e}W(u) \right)\,dxdy\\
    =& \int_{y\in M} \cos\theta_2 \left\{ \int_{-\sqrt{1-y^2}}^{x_0} +\int_{x_0}^{\sqrt{1-y^2}}\right\} \left(\f{\e}{2\cos\theta_2}|\pa_{x}u|^2+ \f{\cos\theta_2}{\e} W(u)\right)\,dx\\
    \geq & (y^*+\cos{\f{\al_3}{2}})\cos\theta_2(\sigma_{13}+\sigma_{23})-C\e^{\f13}.
     \end{split}
\end{equation}

Now proceeding as before by first adding \eqref{ene: om2 vertical} and \eqref{ene: om2 horizontal} and then maximizing the lower bound with respect to $\theta_2$ we obtain
\beq\label{ene:om2}
\begin{split}
&\int_{\Om_2}\left( \f{\e}{2}|\na u|^2+\f{1}{\e}W(u) \right)\,dxdy\\
\geq & \sqrt{\left[\sin{\f{\al_3}{2}}(\sigma_{13}+\sigma_{23})+(\mu_2-\mu_1+\sin{\f{\al_2-\al_1}{2}})(\sigma_{23}-\sigma_{13})\right]^2+\left[(y^*+\cos{\f{\al^3}{2}})(\sigma_{13}+\sigma_{23})\right]^2}-C\e^{\f13}. 
\end{split}
\eeq

Finally we can add up \eqref{ene:om1 final} and \eqref{ene:om2} to get
\beq\label{ene:total}
\begin{split}
    &\int_{B_1} \left(\f{\e}{2}|\na u|^2+\f{1}{\e}W(u)\right)\,dz\\
    \geq & \sqrt{\left[\sin{\f{\al_3}{2}}(\sigma_{13}+\sigma_{23})+(\mu_2-\mu_1+\sin{\f{\al_2-\al_1}{2}})(\sigma_{23}-\sigma_{13})\right]^2+\left[(y^*+\cos{\f{\al^3}{2}})(\sigma_{13}+\sigma_{23})\right]^2}\\
    &\quad +\sigma_{12}\sqrt{(\mu_1+\mu_2)^2+(\cos{\f{\al_2-\al_1}{2}}-y^*)^2}-C\e^{\f13}\\
    =:& E(\mu_1,\mu_2,y^*)-C\e^{\f13}.
\end{split}
\eeq
Then we examine the following minimization problem 
\beq\label{min problem}
\begin{split}
    &\qquad \min{\ E(\mu_1,\mu_2,y^*)} \\
    \text{ such that } &\mu_1\in[0,\sin{\f{\al_3}{2}}+\sin{\f{\al_2-\al_1}{2}}-2c_0\e],\\
    &\mu_2\in[0,\sin{\f{\al_3}{2}}-\sin{\f{\al_2-\al_1}{2}}-2c_0\e],\\ &y_*\in [-\cos{\f{\al_3}{2}}+c_0\e,\cos{\f{\al_2-\al_1}{2}}-c_0\e].
\end{split}
\eeq
Direct calculation implies that 
\beqo
\qquad \min{\ E(\mu_1,\mu_2,y^*)}=\sigma_{12}+\sigma_{13}+\sigma_{23},
\eeqo
and the minimum is reached when
\beqo
\mu_1=\sin{\f{\al_2-\al_1}{2}},\ \mu_2=0,\ y^*=0.
\eeqo
The calculation is straightforward and is provided in the Appendix \ref{app: min prob}. Therefore we have shown that \eqref{lower bound e 1/3} holds when $y^*\leq \cos\f{\al_2-\al_1}{2}-c_0\e$. 
\end{case}

\begin{case}
$y^*>\cos\f{\al_2-\al_1}{2}-c_0\e$. We will show in this case the total energy is strictly larger than $\sigma_{12}+\sigma_{13}+\sigma_{23}$. We write
\beqo
W=\sin^2\theta_3W+\cos^2\theta_3 W,\quad \text{for some }\theta_3\in[0,\f{\pi}{2}].
\eeqo
By the boundary condition \eqref{def of g_eps},
\begin{align*}
    &u((x,\sqrt{1-x^2}))=a_1,\ u((x,-\sqrt{1-x^2}))=a_3,\quad \forall x\in [-\sin{\f{\al_3}{2}}+c_0\e,-\sin{\f{\al_2-\al_1}{2}}-c_0\e],\\
    &u((x,\sqrt{1-x^2}))=a_2,\ u((x,-\sqrt{1-x^2}))=a_3,\quad \forall x\in [-\sin{\f{\al_2-\al_1}{2}}+c_0\e,\sin{\f{\al_3}{2}}-c_0\e],
\end{align*}
which allows us to get the following estimate for the energy in the vertical direction:
\begin{equation}\label{ene:case2,vertical}
    \begin{split}
        &\int_{B_1}\left(\f{\e}{2}|\pa_y u|^2+\f{\sin^2\theta_3}{\e}W(u)\right)\,dz\\
        \geq & \int_{-\sin{\f{\al_3}{2}}+c_0\e}^{-\sin{\f{\al_2-\al_1}{2}}-c_0\e} \sigma_{13}\sin\theta_3\,dx+\int_{-\sin{\f{\al_2-\al_1}{2}}+c_0\e}^{\sin{\f{\al_3}{2}}-c_0\e} \sigma_{23}\sin\theta_3\,dx\\
        \geq & \sin\theta_3\left[ \left( \sin{\f{\al_3}{2}}- \sin{\f{\al_2-\al_1}{2}}\right)\sigma_{13} +\left( \sin{\f{\al_3}{2}}+ \sin{\f{\al_2-\al_1}{2}}\right)\sigma_{23} \right]-C\e.
    \end{split}
\end{equation}

For the horizontal direction, utilizing similar arguments as in Case 1 we have 
\beqo
\cl(\tilde{M})\geq \cos{\f{\al_3}{2}}+\cos{\f{\al_2-\al_1}{2}}-C\e^{\f13},
\eeqo
where $\tilde{M}:=M\cap[-\cos{\f{\al_3}{2}}+c_0\e,\cos{\f{\al_2-\al_1}{2}}-c_0\e]$. For any $y\in\tilde{M}$, 
\begin{equation*}
    \int_{\g_y} \left(\f{\e}{2}|\pa_x u|^2+\f{\cos^2\theta_3}{\e}W(u)\right)\,dx\geq \cos\theta_3 (\sigma_{13}+\sigma_{23})-C\e^{\f13}.
\end{equation*}
Integrating this with respect to $y$ gives
\begin{equation}\label{ene:case2, horizontal}
    \int_{B_1}\left(\f{\e}{2}|\pa_x u|^2+\f{\cos^2\theta_3}{\e}W(u)\right)\,dz\geq \cos\theta_3 (\sigma_{13}+\sigma_{23})(\cos{\f{\al_3}{2}}+\cos{\f{\al_2-\al_1}{2}})-C\e^{\f13}.
\end{equation}
Adding \eqref{ene:case2,vertical} and \eqref{ene:case2, horizontal} together and maximizing with respect to $\theta_3$ gives
\begin{equation}\label{ene:case2,total}
    \begin{split}
        &\int_{B_1}\left(\f{\e}{2}|\na u|^2+\f{1}{\e}W(u)\right)\,dz\\
        \geq &\bigg(\left[ (\sin{\f{\al_3}{2}}- \sin{\f{\al_2-\al_1}{2}})\sigma_{13} +( \sin{\f{\al_3}{2}}+ \sin{\f{\al_2-\al_1}{2}})\sigma_{23} \right]^2 \\
        &\qquad\qquad\quad +\left[(\sigma_{13}+\sigma_{23})(\cos{\f{\al_3}{2}}+\cos{\f{\al_2-\al_1}{2}})\right]^2\bigg)^{\f12}-C\e^{\f13}\\
        >& \sigma_{12}+\sigma_{13}+\sigma_{23}-C\e^{\f13}.
    \end{split}
\end{equation}
The final inequality can be obtained through direct calculation, as presented in the Appendix \ref{app:case2 ineq}. This completes the proof of the proposition.
\end{case}

\end{proof}

The power $\f13$ can be improved to $\f12$.

\begin{proposition}[lower bound of order $\e^{\f12}$] \label{prop epsilon 1/2}
There exist constants $C(W)$ and $\e_1$, such that for any $\e\leq \e_1$, it holds 
\beq\label{lower bound e 1/2}
\int_{B_1} \left(\f{\e}{2}|\na u_\e|^2+\f{1}{\e}W(u_\e)\right)\,dz\geq \sigma_{12}+\sigma_{13}+\sigma_{23}-C\e^{\f12}.
\eeq 
\end{proposition}

\begin{proof}
    Firstly, we will refine some definitions of the previously used quantities in the proof of Proposition \ref{prop: lower bound epsilon 1/3}. The definitions of $\zeta(x),\; \mu_1,\;\mu_2,\;\Om_1,\;\Om_2$ remain unchanged; in the definitions of $\lam_i(y),\; K,\;K_1^+,\;K_2^-$, we replace $\e^{\f16}$ by $\e^{\f14}$; and we redefine $y^*$ as
    \beq\label{redef: y*}
   y^*:=\min\{y\in [-\cos{\f{\al_3}{2}}+c_0\e,1]:\lam_1(y)+\lam_2(y)\geq \cl(\g_y)-\al\e^{\f12}\},
    \eeq
    for some constant $\al$ determined later. 

    When $y^*>\cos{\f{\al_2-\al_1}{2}}-c_0\e$, by the proof of Proposition \ref{prop: lower bound epsilon 1/3} (Case 2) the total energy is strictly larger than $\sigma_{12}+\sigma_{13}+\sigma_{23}$ when $\e$ is small enough, so we may assume $y^*\leq \cos{\f{\al_2-\al_1}{2}}-c_0\e$.

    In $\Om_1$, following similar estimates as in \eqref{ene:om1,vertical} and \eqref{ene:om1,horizontal} we get
    \beq\label{refined ene: om1}
    \begin{split}
          &\int_{\Om_1}\left(\f{\e}2|\na u|^2+\f{1}{\e}W(u)\right)\,dxdy\\
        \geq & \sin\theta_1(\mu_1+\mu_2)(\sigma_{12}-C\e^{\f12})+\cos\theta_1\sigma_{12}(\cos{\f{\al_2-\al_1}{2}}-c_0\e-y^*)
    \end{split}
    \eeq
    for any $\theta_1\in[0,\f{\pi}{2}]$.

    In $\Om_2$, thanks to the analogous estimates as in \eqref{ene: om2 vertical} and \eqref{ene: om2 horizontal},
    \beq\label{refined ene: om2}
    \begin{split}
      &\int_{\Om_2}\left(\f{\e}2|\na u|^2+\f{1}{\e}W(u)\right)\,dxdy\\
        \geq &  \int_{x\in K} \int_{-\sqrt{1-x^2}}^{\zeta(x)} \left( \f{\e}2|\pa_y u|^2+\f{\sin^2\theta_2}{\e} W(u) \right)\,dydx+\int_{y\in M} \int_{\g_y}  \left(  \f\e2|\pa_{x}u|^2+\f{\cos^2\theta_2}{\e}W(u) \right)\,dxdy\\
        &\qquad +\int_{y\in [-\cos\f{\al_3}{2}+c_0\e,y^*]\setminus M}\int_{\g_y} \f{\cos^2\theta_2}{\e} W(u)\,dxdy\\
        \geq & \sin\theta_2\left[ (\sin\f{\al_3}{2}-\al \e^{\f12})(\sigma_{13}+\sigma_{23})+(\mu_2-\mu_1+\sin{\f{\al_2-\al_1}{2}})(\sigma_{23}-\sigma_{13}) \right] -O(\e)\\
        &\qquad + (y^*+\cos\f{\al_3}{2}-\beta)\cos\theta_2(\sigma_{13}+\sigma_{23})-C\e^{\f12} +\f12 c_W \al \cos^2\theta_2\beta,  \qquad \text{ for any }\theta_2\in [0,\f{\pi}{2}]. 
    \end{split}
    \eeq
    where $\beta:=y^*+\cos\f{\al_3}{2}-\cl(M)\geq 0$. Substituting $\theta_1=\f{\al_2-\al_1}{2}$ in \eqref{refined ene: om1} and $\theta_2=\f{\al_3}{2}$ in \eqref{refined ene: om2} and adding them up, 
    \begin{equation}\label{refined ene: total}
        \begin{split}
            &\int_{B_1}\left(\f{\e}2|\na u|^2+\f{1}{\e}W(u)\right)\,dxdy\\
            \geq & \left[ \sin\f{\al_2-\al_1}{2}(\mu_1+\mu_2)\sigma_{12}+(\mu_2-\mu_1)\sin\f{\al_3}{2}(\sigma_{23}-\sigma_{13}) \right]\\
            \qquad &+\left[\cos^2\f{\al_2-\al_1}{2}\sigma_{12}+\sin\f{\al_3}{2}(\sigma_{23}-\sigma_{13})\sin\f{\al_2-\al_1}{2}\right]\\
            \qquad &+(\sigma_{13}+\sigma_{23})+C_1(\al-C_2)\beta-C_3(\al+1)\e^{\f12}-C_4\e\\
            =&(\sigma_{12}+\sigma_{13}+\sigma_{23})+2\mu_2\sin\f{\al_2-\al_1}{2}\sigma_{12}+C_1(\al-C_2)\beta-C_3(\al+1)\e^{\f12}-C_4\e,
        \end{split}  
    \end{equation}
    where we have used \eqref{sum of angles} and \eqref{relation of angles} repeatedly. Here all the constants $C_i$ ($i=1,...,4$) only depend on the potential $W$. We can take $\al=2C_2$ in \eqref{refined ene: total} to obtain 
    \beq
    \int_{B_1}\left(\f{\e}2|\na u|^2+\f{1}{\e}W(u)\right)\,dxdy\geq (\sigma_{12}+\sigma_{13}+\sigma_{23})-C\e^{\f12}
    \eeq
    for some constant $C=C(W)$ and sufficiently small $\e$. 
\end{proof}
\begin{rmk}
From the lower bound \eqref{refined ene: total} and the upper bound \eqref{upper bound}, we have 
\beqo
\beta=y^*+\cos\f{\al_3}{2}-\cl(M)\leq C\e^{\f12}.
\eeqo
Then one can examine the minimization problem \eqref{new min problem}  studied in Appendix \ref{app: min prob} to conclude that for the miminizer $u_\e$,
\beq\label{loc of center}
y^*\leq C\e^{\f14}, \quad \text{for some }C=C(W).
\eeq
For a more detailed argument and proof, please refer to \cite[Lemma 4.1]{alikakos2022triple}. 

\end{rmk}

\section{Proof of Theorem \ref{main theorem}}
The remaining proof of Theorem \ref{main theorem} closely follows the proof of \cite[Theorem 1.2]{alikakos2022triple}. Particularly, the results concerning the localization of the diffuse interface (\cite[Proposition 4.2, Proposition 5.1]{alikakos2022triple}), the existence of a sufficient number of triplets of points in close proximity to the three distinct phases (\cite[Lemma 6.3]{alikakos2022triple}), and the analysis of the minimizer's blow-up/blow-down limit (\cite[Section 6]{alikakos2022triple}) all apply to our case with minor modifications.

There is one crucial point that requires our attention. In \cite[Proposition 4.2]{alikakos2022triple}, the proof relies on the fact that the part of the lower bound over $\Omega_1$ in \cite[Proposition 3.2]{alikakos2022triple} is based solely on the horizontal gradient. This allows the vertical gradient to be added to produce an improvement. However, in our case, we already take both the horizontal and vertical gradients into account when estimating the energy in $\Om_1$.  To overcome this issue, we can rotate the coordinate system in $\Om_1$ such that the diffuse interface is parallel to the new $y$-axis, and then calculate the energy in the new coordinate. To be more specific, we examine the proof of Proposition \ref{prop: lower bound epsilon 1/3}. In $\Omega_1$, we can rotate the coordinate through an angle $\f{\al_2-\al_1}{2}$ by setting
\beqo
\tilde{x}=\cos\f{\al_2-\al_1}{2}x+\sin\f{\al_2-\al_1}{2}y,\quad \tilde{y}=-\sin\f{\al_2-\al_1}{2}x+\cos\f{\al_2-\al_1}{2}y.
\eeqo
Then utilizing the boundary condition, we can get exactly the same lower bound in $\Om_1$ as in \eqref{ene:om1} by only integrating the energy with the partial gradient $\left|\f{\pa u}{\pa\tilde{x}}\right|^2$ along the lines parallel to $\tilde{x}$-axis. Then we can argue in the same way as \cite[Proposition 4.2]{alikakos2022triple} to conclude that the diffuse interface has to stay in a $C\e^{\f14}$ neighborhood of the ``triod".

\section*{Competing interests}
The authors declare that they have no conflict of interest.

\section*{Acknowledgement}
Z. Geng is supported by the Basque Government through the BERC 2022--2025 program and by the Spanish State Research Agency through BCAM Severo Ochoa excellence accreditation SEV--2017--0718 and through project PID2020--114189RB--I00 funded by Agencia Estatal de Investigaci\'{o}n (PID2020--114189RB--I00/AEI/10.13039/501100011033).

\section*{Data Availability Statement}

Data sharing is not applicable to this article as no datasets were generated or analysed during the current study.

\appendix

\section{The minimization problem \eqref{min problem}} \label{app: min prob}
Recall that
\begin{equation*}
    \begin{split}
        &E(\mu_1,\mu_2,y^*)\\
        =& \sqrt{\left[\sin{\f{\al_3}{2}}(\sigma_{13}+\sigma_{23})+(\mu_2-\mu_1+\sin{\f{\al_2-\al_1}{2}})(\sigma_{23}-\sigma_{13})\right]^2+\left[(y^*+\cos{\f{\al^3}{2}})(\sigma_{13}+\sigma_{23})\right]^2}\\
    &\quad +\sigma_{12}\sqrt{(\mu_1+\mu_2)^2+(\cos{\f{\al_2-\al_1}{2}}-y^*)^2}.
    \end{split}
\end{equation*}
An easy observation is that $E(\mu_1,\mu_2,y^*)$ reduces if we replace $(\mu_1,\mu_2)$ by $(\mu_1,0)$. Thus we can assume $\mu_2=0$. Set
\beqo
\mu^*:=\sin\f{\al_2-\al_1}{2}-\mu_1. 
\eeqo
By \eqref{relation of angles} we consider the following equivalent minimization problem
\beq\label{new min problem}
\begin{split}
    &\qquad \min{\ \tilde{E}(\mu^*,y^*)} \\
    \text{ such that } &\mu^*\in[-\sin{\f{\al_3}{2}}+2c_0\e,\sin\f{\al_2-\al_1}{2}],\\
    &y_*\in [-\cos{\f{\al_3}{2}}+c_0\e,\cos{\f{\al_2-\al_1}{2}}-c_0\e],
\end{split}
\eeq
where 
\begin{equation*}
    \begin{split}
        &\tilde{E}(\mu^*,y^*)\\
        =& \sqrt{\left[\sin{\f{\al_3}{2}}(\sin\al_1+\sin\al_2)+\mu^*(\sin\al_1-\sin\al_2)\right]^2+\left[(y^*+\cos{\f{\al_3}{2}})(\sin\al_1+\sin\al_2)\right]^2}\\
    &\quad +\sin\al_3\sqrt{(\sin\f{\al_2-\al_1}{2}-\mu^*)^2+(\cos{\f{\al_2-\al_1}{2}}-y^*)^2}.
    \end{split}
\end{equation*}

The objective is to prove that $\tilde{E}$ obtains its minimal value $\sin\al_1+\sin\al_2+\sin\al_3$ if and only if $y^*=\mu^*=0$.

By \eqref{sum of angles}, 
\beqo
\sin{\f{\al_1+\al_2}{2}}=\sin{\f{\al_3}{2}},\quad \cos{\f{\al_1+\al_2}{2}}=-\cos{\f{\al_3}{2}}.
\eeqo

Using this and the Cauchy Schwartz inequality $\sqrt{a^2+b^2}\sqrt{c^2+d^2}\geq ac+bd$, we estimate
\begin{align*}
    &\tilde{E}(\mu^*,y^*)\\
    \geq & \sin^2{\f{\al_3}{2}}(\sin\al_1+\sin\al_2)+\mu^*\sin\f{\al_3}{2}(\sin\al_1-\sin\al_2)+\cos\f{\al_3}{2}(y^*+\cos\f{\al_3}{2})(\sin\al_1+\sin\al_2)\\
    &\quad + \sin\al_3\left( (\sin\f{\al_2-\al_1}{2}-\mu^*)\sin\f{\al_2-\al_1}{2}+(\cos\f{\al_2-\al_1}{2}-y^*)\cos\f{\al_2-\al_1}{2} \right)\\
    =&\sin\al_1+\sin\al_2+\sin\al_3.
\end{align*}
The equality holds if and only if 
\begin{align}
\label{eq cond 1}\f{\sin{\f{\al_3}{2}}(\sin\al_1+\sin\al_2)+\mu^*(\sin\al_1-\sin\al_2)}{(y^*+\cos\f{\al_3}{2})(\sin\al_1+\sin\al_2)}&=\f{\sin\f{\al_3}{2}}{\cos\f{\al_3}{2}},\\
\label{eq cond 2}\f{\sin\f{\al_2-\al_1}{2}-\mu^*}{\cos\f{\al_2-\al_1}{2}-y^*}&=\f{\sin\f{\al_2-\al_1}{2}}{\cos\f{\al_2-\al_1}{2}}.
\end{align}

Suppose $y^*\neq 0$, then by \eqref{eq cond 2},
\beqo
\f{\mu^*}{y^*}=\f{\sin\f{\al_2-\al_1}{2}}{\cos\f{\al_2-\al_1}{2}}.
\eeqo
Substituting this into \eqref{eq cond 1} implies
\begin{align*}
    &\f{\sin\f{\al_2-\al_1}{2}(\sin\al_1-\sin\al_2)}{\cos\f{\al_2-\al_1}{2}(\sin\al_1+\sin\al_2)}=\f{\sin\f{\al_3}{2}}{\cos\f{\al_3}{2}}\\
    \Rightarrow & \tan^2\f{\al_2-\al_1}{2}=\tan^2\f{\al_3}{2},
\end{align*}
which is impossible since $0\leq \f{\al_2-\al_1}{2}<\f{\al_3}{2}<\f{\pi}{2}$. Therefore in order to get the equality it must hold that $y^*=0$. Then by \eqref{eq cond 2} $\mu^*=0$. The proof is completed.

\section{The last inequality in \eqref{ene:case2,total}} \label{app:case2 ineq}
The objective is to show
\begin{equation*}
    \begin{split}
       ( \sigma_{12}+\sigma_{13}+\sigma_{23})^2< &\left[ (\sin{\f{\al_3}{2}}- \sin{\f{\al_2-\al_1}{2}})\sigma_{13} +( \sin{\f{\al_3}{2}}+ \sin{\f{\al_2-\al_1}{2}})\sigma_{23} \right]^2+ \\
        &\qquad\quad +\left[(\sigma_{13}+\sigma_{23})(\cos{\f{\al_3}{2}}+\cos{\f{\al_2-\al_1}{2}})\right]^2
    \end{split}
\end{equation*}
Or by \eqref{relation of angles} equivalently,
\begin{equation*}
    \begin{split}
       ( \sin{\al_1}+\sin{\al_2}+\sin{\al_3})^2< &\left[ (\sin{\f{\al_3}{2}}- \sin{\f{\al_2-\al_1}{2}})\sin{\al_2} +( \sin{\f{\al_3}{2}}+ \sin{\f{\al_2-\al_1}{2}})\sin{\al_1} \right]^2+ \\
        &\qquad\quad +\left[(\sin{\al_2}+\sin{\al_1})(\cos{\f{\al_3}{2}}+\cos{\f{\al_2-\al_1}{2}})\right]^2
    \end{split}
\end{equation*}

We directly compute
\begin{align*}
    &\left[ \sin{\f{\al_3}{2}}(\sin\al_2+\sin\al_1)+\sin{\f{\al_2-\al_1}{2}}(\sin\al_1-\sin\al_2) \right]^2\\
    &\qquad\quad +\left[(\sin\al_1+\sin\al_2)(\cos{\f{\al_3}{2}}+\cos{\f{\al_2-\al_1}{2}})\right]^2-( \sin{\al_1}+\sin{\al_2}+\sin{\al_3})^2\\
    =&4\left(\sin^2\f{\al_3}{2}\cos\f{\al_2-\al_1}{2}+\sin^2\f{\al_2-\al_1}{2}\cos\f{\al^3}{2}\right)^2\\
    &\qquad +4\sin^2\f{\al_3}{2}\cos^2\f{\al_2-\al_1}{2}\left( \cos\f{\al_3}{2} +\cos\f{\al_2-\al_1}{2}\right)^2-4\sin^2\f{\al_3}{2}\left( \cos\f{\al_3}{2} +\cos\f{\al_2-\al_1}{2}\right)^2\\
    =& 4\left(\sin^2\f{\al_3}{2}\cos\f{\al_2-\al_1}{2}+\sin^2\f{\al_2-\al_1}{2}\cos\f{\al^3}{2}\right)^2\\
    &\qquad -4\sin^2\f{\al_3}{2}\sin^2\f{\al_2-\al_1}{2}\left( \cos\f{\al_3}{2} +\cos\f{\al_2-\al_1}{2}\right)^2\\
    =&4\left(\sin^2\f{\al_3}{2}-\sin^2\f{\al_2-\al_1}{2}\right)^2\\
    >&0.
\end{align*}

\bibliographystyle{acm}
\bibliography{bib_triple_junction-unequal}

\begin{thebibliography}{1}

\bibitem{AF}
{\sc Alikakos, N.~D., and Fusco, G.}
\newblock {Sharp lower bounds for vector Allen-Cahn energy and qualitative
  properties of minimizes under no symmetry hypotheses}.
\newblock {\em arXiv preprint arXiv:2110.00388, to appear in the Bulletin of
  the Hellenic Mathematical Society\/} (2021).

\bibitem{afs-book}
{\sc Alikakos, N.~D., Fusco, G., and Smyrnelis, P.}
\newblock {\em Elliptic systems of phase transition type}, vol.~91 of {\em
  Progress in Nonlinear Differential Equations and Their Applications}.
\newblock Springer-Birkh\"{a}user, 2018.

\bibitem{alikakos2022triple}
{\sc Alikakos, N.~D., and Geng, Z.}
\newblock On the triple junction problem on the plane without symmetry
  hypotheses.
\newblock {\em arXiv preprint arXiv:2210.17134\/} (2022).

\bibitem{white1996existence}
{\sc White, B.}
\newblock Existence of least-energy configurations of immiscible fluids.
\newblock {\em The Journal of Geometric Analysis 6}, 1 (1996), 151--161.

\end{thebibliography}

\end{document}